\documentclass[11pt,twoside]{article}
\usepackage{theorem}
\usepackage{epsfig,amsmath,amssymb}
\usepackage{color}
\usepackage{lineno}
\allowdisplaybreaks[1]

\newtheorem{thm}{\bf Theorem}[section]
\newtheorem{cor}[thm]{\bf Corollary}
\newtheorem{lem}[thm]{\bf Lemma}

\newtheorem{defn}[thm]{\bf Definition}
\theoremstyle{remark}

\numberwithin{equation}{section}
\newtheorem{exam}[thm]{\bf Example}

\title{\bf Stability of a Fr\'{e}chet-type functional equation}

\author{Sun-Sook Jin, Soon-Mo Jung and Yang-Hi Lee}
\date{\footnotesize\it
      Department of Mathematics Education,
      Gongju National University of Education,
      Gongju 32553, Republic of Korea \\
      \textit{E-mail}: {\tt ssjin@gjue.ac.kr} \\
      Mathematics Section,
      College of Science and Technology,
      Hongik University,
      30016 Sejong, \\
      Republic of Korea \\
      \textit{E-mail}: {\tt smjung747@gmail.com} \\
      Department of Mathematics Education,
      Gongju National University of Education,
      Gongju 32553, Republic of Korea \\
      \textit{E-mail}: {\tt lyhmzi@gjue.ac.kr}}

\begin{document}
\maketitle
\markboth{}{}

\small
\vspace{1cm}

\noindent
{\bf Abstract.}
We establish the generalized Hyers-Ulam stability of the
Fr\'{e}chet-type functional equation
\begin{align*}
\sum_{\emptyset \neq I \subset \{ 1, 2, \ldots, n+1 \}}
(-1)^{|I|-n+1} f \Bigg( \sum_{i \in I} x_i \Bigg) = 0
\end{align*}
in the sense of G\u{a}vru\c{t}a.
\vspace*{5mm}

\vspace*{5mm}

\normalsize

\section{Introduction}

Throughout this paper, let $V$ and $W$ be real vector spaces,
$X$ a real normed space, $Y$ a real Banach space, and let $n$
be a fixed positive integer.
For a given mapping $f : V \to W$, we define the
Fr\'{e}chet-type difference operator
$E_n f(x_1, x_2, \ldots, x_{n+1})$ by
\begin{align*}
E_n f(x_1, x_2, \ldots, x_{n+1})
:= \sum_{\emptyset \neq I \subset \{ 1, 2, \ldots, n+1 \}}
   (-1)^{|I|-n+1} f \Bigg( \sum_{i \in I} x_i \Bigg)
\end{align*}
for all $x_1, x_2, \ldots, x_{n+1} \in V$.

The stability theory of functional equations originated from
the pioneering work of Hyers \cite{Hy} concerning the Ulam's
problem \cite{Ul} for the additive functional equation, which
can be expressed in terms of the operator $E_1$ as
\begin{align*}
E_1 f(x_1, x_2) = 0
\end{align*}
for all $x_1, x_2 \in V$.
This foundational result was subsequently extended by Rassias
\cite{ThR} and G\u{a}vru\c{t}a \cite{Gav}.
Since then, their contributions have established what are now
standardly termed \emph{Hyers-Ulam-Rassias stability} and
\emph{generalized Hyers-Ulam stability}, respectively, in the
literature.

In particular, several researchers
\cite{Changqa,Fechnerqa,Jinqa,Jungqa,Kangqa,Kannappanqa,Kimqa,Leeqa}
have investigated the stability of the additive-quadratic
functional equation, which corresponds to the case where
\begin{align*}
E_2 f(x_1, x_2, x_3) = 0
\end{align*}
for all $x_1, x_2, x_3 \in V$.
Moreover, Jin and Lee \cite{Jinqcqa} extended this line of
research by investigating the stability of the
additive-quadratic-cubic-quartic functional equation
\begin{align*}
E_4 f(x_1, x_2, x_3, x_4, x_5) = 0
\end{align*}
for all $x_1, x_2, x_3, x_4, x_5 \in V$.

In this paper, we establish the generalized Hyers-Ulam
stability of the Fr\'{e}chet-type functional equation
\begin{align} \label{i:1}
E_n f(x_1, x_2, \ldots, x_{n+1}) = 0
\end{align}
in the sense of G\u{a}vru\c{t}a.
As direct corollaries of our main results, we also obtain its
Hyers-Ulam-Rassias stability and hyperstability.

\section[Stability of Fr\'{e}chet-type functional equation]
        {Stability of Fr\'{e}chet-type functional equation \\
         $E_n f(x_1, x_2, \ldots, x_{n+1}) = 0$}

By choosing $a \geq 2$ and setting
$A = \{ a, a^2, \ldots, a^n \}$, we obtain the following
specialized cases of \cite[Definition 2.1]{JinChar}.


\begin{defn}[\cite{JinChar}] \label{Enan:def1}
For each $k \in \{ 1, 2, \ldots, n \}$, let $\omega_{k,a,n}$
be the real number defined by
\begin{align*}
\omega_{k,a,n}
:= \prod\limits_{\substack{1 \leq i \leq n \\ i \neq k}}
   \frac{1}{a^k - a^i}.
\end{align*}
For a given mapping $f : V \to W$, we define the component
mappings $f_{k,a,n}, \Gamma_{a,n} f : V \to W$ by
\begin{align*}
\begin{array}{c}
{\displaystyle
f_{k,a,n}(x)
:= \omega_{k,a,n}
   \sum_{I \subset \{ 1, 2, \ldots, n \} \setminus \{ k \}}
   \Bigg( (-1)^{|I|} \prod_{i \in I} a^i \Bigg)
   f \big( a^{n-1-|I|} x \big)}, \vspace*{2mm}\\
{\displaystyle
\Gamma_{a,n} f(x)
:= \sum_{I \subset \{ 1, 2, \ldots, n \}}
   \Bigg( (-1)^{|I|} \prod_{i \in I} a^i \Bigg)
   f \big( a^{n-|I|} x \big)}
\end{array}
\end{align*}
for all $x \in V$ and $k \in \{ 1, 2, \ldots, n \}$.
\end{defn}

The following lemma generalizes \cite[Theorem 11]{JinGp8} from
$a = 2$ to $a \geq 2$.
Since the proof is identical to that of
\cite[Theorem 11]{JinGp8}, it is omitted.


\begin{lem}[\cite{JinGp8}] \label{Enan:lem2}
The following identity holds for all
$i \in \{ 0, 1, \ldots, n-2 \}$:
\begin{align} \label{Enan:1}
\sum_{k=1}^n a^{ki} \omega_{k,a,n} = 0.
\end{align}
Furthermore, for any fixed $\ell \in \{ 1, 2, \ldots, n \}$
and all integers $i \ge 0$, the following inequalities hold:
\begin{align} \label{Enan:2}
\Bigg| \sum_{k=\ell}^n \frac{\omega_{k,a,n}}{a^{k(i+1)}} \Bigg|
\leq \bigg| \frac{\omega_{\ell,a,n}}{a^{\ell(i+1)}} \bigg|
\end{align}
and
\begin{align} \label{Enan:3}
\Bigg| \sum_{k=1}^\ell a^{ki} \omega_{k,a,n} \Bigg|
\leq a^{\ell i} | \omega_{\ell,a,n} |.
\end{align}
\end{lem}

By setting $a_k = a^k$ in \cite[Corollary 6]{JungU}, we obtain
the following uniqueness lemma.


\begin{lem}[\cite{JungU}] \label{Enan:lem3}
Let $f : V \to Y$ be a given mapping and let
$\ell \in \{ 0, 1, \ldots, n \}$.
Assume that there exist a mapping $F : V \to Y$ and a function
$\phi : V \setminus \{ 0 \} \to [0, \infty)$ such that the
inequality
\begin{align} \label{Enan:4}
\| f(x) - F(x) \| \leq \left\{
\begin{array}{ll}
{\displaystyle
\sum_{i=0}^\infty \frac{\phi(a^i x)}{a^{i+1}} < \infty}
& (\mbox{for}~ \ell = 0), \vspace*{2mm}\\
{\displaystyle
\sum_{i=0}^\infty \frac{\phi(a^i x)}{a^{(\ell+1)(i+1)}} +
\sum_{i=0}^\infty
a^{\ell i} \phi \bigg( \frac{1}{a^i} x \bigg) < \infty}
& (\mbox{for}~ 0 < \ell < n), \vspace*{2mm}\\
{\displaystyle
\sum_{i=0}^\infty a^{ni} \phi \bigg( \frac{1}{a^i} x \bigg)
< \infty}
& (\mbox{for}~ \ell = n)
\end{array}
\right.
\end{align}
holds for all $x \in V \setminus \{ 0 \}$.
If $F$ admits a decomposition
$F(x) = \sum\limits_{k=1}^n F_k(x)$ where each $F_k : V \to Y$
satisfies $F_k(ax) = a^k F_k(x)$ for all $x \in V$, then $F$
is the unique mapping with this property satisfying
$(\ref{Enan:4})$.
\end{lem}

Setting $a_k = a^k$ in \cite[Theorem 2.5]{JinChar} yields the
following decomposition lemma.


\begin{lem} \label{Enan:lem4}
For any mapping $f : V \to W$, the identity
\begin{align*}
f(x) = \sum_{k=1}^n f_{k,a,n}(x)
\end{align*}
holds for all $x \in V$, where each $f_{k,a,n}$ is defined as
in Definition \ref{Enan:def1}.
\end{lem}

The following lemma presents a special case of
\cite[Theorem 2.11]{JinChar}, adapted for
$A = \{ a, a^2, \ldots, a^n \}$.


\begin{lem}[\cite{JinChar}] \label{Enan:lem5}
A mapping $f : V \to W$ satisfies the functional equation
\begin{align*}
\Gamma_{a,n} f(x) = 0
\end{align*}
for all $x \in V$ if and only if there exist mappings
$f_1, f_2, \ldots, f_n : V \to W$ such that
\begin{align*}
f(x) = \sum_{k=1}^n f_k(x) ~~~\mbox{and}~~~
f_k(ax) = a^k f_k(x)
\end{align*}
for all $x \in V$ and each $k \in \{ 1, 2, \ldots, n \}$.
In this case, each $f_k$ coincides with $f_{k,a,n}$ given in
Definition \ref{Enan:def1}, which yields the explicit
representation
\begin{align} \label{Enan:5}
f(x) = \sum_{k=1}^n f_{k,a,n}(x)
\end{align}
for all $x \in V$.
\end{lem}

Applying \cite[Theorem 42]{Ch} yields the following stability
result.


\begin{thm} \label{Enan:thm6}
For a fixed $\ell \in \{ 0, 1, \ldots, n \}$ and a function
$\varphi : V \setminus \{ 0 \} \to [0,\infty)$, suppose that
the following conditions hold for all
$x \in V \setminus \{ 0 \}$:
\begin{align} \label{Enan:6}
\left\{
\begin{array}{ll}
{\displaystyle
\sum_{i=0}^\infty \frac{\varphi(a^i x)}{a^{i+1}} < \infty}
& (\mbox{for}~ \ell = 0), \vspace*{2mm}\\
{\displaystyle
\sum_{i=0}^\infty \frac{\varphi(a^i x)}{a^{(\ell+1)(i+1)}} +
\sum_{i=0}^\infty
a^{\ell i} \varphi \bigg( \frac{1}{a^i} x \bigg) < \infty}
& (\mbox{for}~ 0 < \ell < n), \vspace*{2mm}\\
{\displaystyle
\sum_{i=0}^\infty a^{ni} \varphi \bigg( \frac{1}{a^i} x \bigg)
< \infty}
& (\mbox{for}~ \ell = n).
\end{array}
\right.
\end{align}
If a mapping $f : V \to Y$ satisfies the inequality
\begin{align*}
\| \Gamma_{a,n} f(x) \| \leq \varphi(x)
\end{align*}
for all $x \in V \setminus \{ 0 \}$, then there exists a
mapping $F : V \to Y$ such that $\Gamma_{a,n} F(x) = 0$ and
\begin{align} \label{Enan:7}
\| f(x) - F(x) \| \leq \left\{
\begin{array}{ll}
{\displaystyle
\sum_{i=0}^\infty
\Bigg| \sum_{k=1}^\ell \omega_{k,a,n} a^{k i} \Bigg|
\varphi \bigg( \frac{1}{a^{i+1}} x \bigg)}
& (\mbox{for}~ 0 \leq \ell < n), \\
\hspace*{5mm}
{\displaystyle
+\sum_{i=0}^\infty
\Bigg| \sum_{k=\ell+1}^n \frac{\omega_{k,a,n}}{a^{k(i+1)}} \Bigg|
\varphi \big( a^i x \big)} \vspace*{2mm}\\
{\displaystyle
\sum_{i=n-1}^\infty
\Bigg| \sum_{k=1}^n \omega_{k,a,n} a^{ki} \Bigg|
\varphi \bigg( \frac{1}{a^{i+1}} x \bigg)}
& (\mbox{for}~ \ell = n)
\end{array}
\right.
\end{align}
for all $x \in V \setminus \{ 0 \}$, where we adopt a
convention $\omega_{0,a,n} = 0$.
In particular, the component mappings $F_{k,a,n} : V \to Y$
for each $k \in \{ 1, 2, \ldots, n \}$ are given by
\begin{align} \label{Enan:8}
F_{k,a,n}(x) = \left\{
\begin{array}{ll}
{\displaystyle
\lim_{i \to \infty}
\frac{1}{a^{ki}} f_{k,a,n} \big( a^i x \big)}
& (\mbox{for}~ k > \ell), \vspace*{2mm}\\
{\displaystyle
\lim_{i \to \infty}
a^{ki} f_{k,a,n} \bigg( \frac{1}{a^i} x \bigg)}
& (\mbox{for}~ k \leq \ell)
\end{array}
\right.
\end{align}
for all $x \in V \setminus \{ 0 \}$, and $F_{k,a,n}(0) = 0$.
\end{thm}

The following theorem is directly obtained from Lemmas
\ref{Enan:lem2}--\ref{Enan:lem4} and Theorem \ref{Enan:thm6}.


\begin{thm} \label{Enan:thm7}
For a fixed $\ell \in \{ 0, 1, \ldots, n \}$ and a function
$\varphi : V \setminus \{ 0 \} \to [0,\infty)$, suppose that
the conditions in $(\ref{Enan:6})$ hold for all
$x \in V \setminus \{ 0 \}$.
If a mapping $f : V \to Y$ satisfies the inequality
\begin{align*}
\| \Gamma_{a,n} f(x) \| \leq \varphi(x)
\end{align*}
for all $x \in V \setminus \{ 0 \}$, then there exists a unique
mapping $F : V \to Y$ such that $\Gamma_{a,n} F(x) = 0$ and
\begin{align} \label{Enan:9}
\| f(x) - F(x) \| \leq \left\{
\begin{array}{ll}
{\displaystyle
| \omega_{\ell,a,n} |
\sum_{i=0}^\infty
a^{\ell i} \varphi \bigg( \frac{1}{a^{i+1}} x \bigg)}
& (\mbox{for}~ 0 \leq \ell < n), \\
\hspace*{5mm}
{\displaystyle
+| \omega_{\ell+1,a,n} |
\sum_{i=0}^\infty \frac{\varphi(a^i x)}{a^{(\ell+1)(i+1)}}}
\vspace*{2mm}\\
{\displaystyle
| \omega_{n,a,n} |
\sum_{i=n-1}^\infty
a^{ni} \varphi \bigg( \frac{1}{a^{i+1}} x \bigg)}
& (\mbox{for}~ \ell = n)
\end{array}
\right.
\end{align}
for all $x \in V \setminus \{ 0 \}$.
In particular, the component mappings $F_{k,a,n} : V \to Y$
for each $k \in \{ 1, 2, \ldots, n \}$ are given by
$(\ref{Enan:8})$ for all $x \in V \setminus \{ 0 \}$, and
$F_{k,a,n}(0) = 0$.
\end{thm}

The following lemma contains an important identity required
for the proof of the main theorem.
Moreover, the identity can be explicitly reconstructed by
tracing the proof of the lemma.


\begin{lem} \label{Enan:lem8}
There exist a natural number $m$, nonzero integers
$b_1, b_2, \ldots, b_m$, and nonzero integers $a_{i,j}$ for
each $i \in \{ 1, 2, \ldots, m \}$ and
$j \in \{ 1, 2, \ldots, n+1 \}$ such that the identity
\begin{align} \label{Enan:10}
\Gamma_{a,n} f(x)
= \sum_{i=1}^m
  b_i E_n f(a_{i,1} x, a_{i,2} x, \ldots, a_{i,{n+1}} x)
\end{align}
holds for an arbitrary mapping $f : V \to Y$ and all $x \in V$.
\end{lem}

\noindent
\textbf{Proof.}
By Definition \ref{Enan:def1}, the operator $\Gamma_{a,n} f(x)$
can be expanded as follows:
\begin{align*}
\Gamma_{a,n} f(x)
& = \sum_{I \subset \{ 1, 2, \ldots, n \}}
    \Bigg( (-1)^{|I|} \prod_{i \in I} a^i \Bigg)
    f \big( a^{n-|I|} x \big) \\
& = f(a^n x) - C_1 f(a^{n-1} x) + C_2 f(a^{n-2} x) - \cdots +
    (-1)^n C_n f(x),
\end{align*}
where each coefficient $C_k$ is the elementary symmetric
polynomial of degree $k$ in the variables
$\{ a^1, a^2, \ldots, a^n \}$, given by
\begin{align*}
C_k = \sum_{1 \leq i_1 < i_2 < \cdots < i_k \leq n}
      a^{i_1 + i_2 + \cdots + i_k}.
\end{align*}

On the other hand, according to the definition of the operator
$E_n$, evaluating it at the points
$(\underbrace{x, \ldots, x}_{n}, (a^n - n) x)$ yields
\begin{align*}
& E_n f \big( \underbrace{x, \ldots, x}_{n}, (a^n - n) x \big)
  \\
& \hspace*{3mm}
  = \sum _{k=0}^n (-1)^{k-n} {n \choose k}
    f \big( (a^n - n + k) x \big) +
    \sum _{j=1}^n (-1)^{j-n+1} {n \choose j} f(jx) \\
& \hspace*{3mm}
  = f(a^n x) - n f \big( (a^n - 1) x \big) +
    {n \choose 2} f \big( (a^n - 2) x \big) - \cdots +
    (-1)^{-n} f \big( (a^n - n) x \big) \\
& \hspace*{8mm}
    -f(nx) + n f \big( (n-1) x \big) -
    {n \choose 2} f \big( (n-2) x \big) + \cdots +
    (-1)^n n f(x).
\end{align*}

By setting $c_{1,0} = 1$, the leading term $f(a^n x)$ is
successfully eliminated from the linear combination
$\Gamma_{a,n} f(x) -
 c_{1,0} E_n f(\underbrace{x, \ldots, x}_{n}, (a^n-n)x)$.
Consequently, the remaining expression can be expressed in the
form $\sum\limits_{i=0}^{a^n-2} c_{2,i} f((a^n-1-i) x)$ for
some integers $c_{2,i}$.

By repeating this elimination process, the term $f((a^n-1) x)$
can be canceled out by choosing the integer $c_{2,0}$.
This ensures that the expression
\begin{align*}
\Gamma_{a,n} f(x) -
c_{1,0}
E_n f \big( \underbrace{x, \ldots, x}_{n}, (a^n - n) x \big) -
c_{2,0} E_n
f \big( \underbrace{x, \ldots, x}_{n}, (a^n - n - 1) x \big)
\end{align*}
reduces to the form
$\sum\limits_{i=0}^{a^n-3} c_{3,i} f((a^n-2-i) x)$ for some
integers $c_{3,i}$.

Continuing this mathematical reduction inductively for all
higher-order terms down to $f((n+1) x)$, we systematically
eliminate all terms involving $f(kx)$ for $k > n$.
Consequently, the remaining lower-order terms can be
represented as a linear combination of $f(ix)$ for
$i \in \{ 1, 2, \ldots, n \}$.
Thus, there exist integers $d_1, d_2, \ldots, d_n$ such that
\begin{align} \label{Enan:11}
\Gamma_{a,n} f(x) - \sum_{j=1}^{a^n-n} c_{j,0}
E_n f \big( \underbrace{x, \ldots, x}_{n}, (a^n-n-j+1) x \big)
= \sum_{i=1}^n d_i f(ix)
\end{align}
holds for any mapping $f : V \to Y$ and all $x \in V$.

Next, we show that $d_i = 0$ for all
$i \in \{ 1, 2, \ldots, n \}$.
Consider the case where $V = Y = \mathbb{R}$.
For each $k \in \{ 1, 2, \ldots, n \}$, we define the monomial
test function $f_k : \mathbb{R} \to \mathbb{R}$ by
$f_k(x) = x^k$ for all $x \in \mathbb{R}$.
Since $f_k$ is a polynomial of degree $k \leq n$, its $(n+1)$th
forward difference vanishes identically, which implies
$E_n f_k(x_1, x_2, \ldots, x_{n+1}) = 0$ for all
$x_1, x_2, \ldots, x_{n+1} \in \mathbb{R}$.
Furthermore, Lemma \ref{Enan:lem4} guarantees that
$\Gamma_{a,n} f_k(x) = 0$ for all
$k \in \{ 1, 2, \ldots, n \}$.

Substituting $f = f_k$ into (\ref{Enan:11}), the left-hand
side vanishes completely, yielding
\begin{align*}
\sum_{i=1}^n d_i f_k(ix) = \sum_{i=1}^n d_i (ix)^k = 0
\end{align*}
for all $x \in \mathbb{R}$.
Evaluating this at $x = 1$, we obtain
\begin{align*}
\sum_{i=1}^n d_i i^k = 0
\end{align*}
for each $k \in \{ 1, 2, \ldots, n \}$.
This leads to the following system of linear equations for the
vector $(d_1, d_2, \ldots, d_n)^T$:
\begin{align*}
\left\{
\begin{array}{l}
d_1 \cdot 1 + d_2 \cdot 2 + \cdots + d_n \cdot n = 0,
\vspace*{2mm}\\
d_1 \cdot 1^2 + d_2 \cdot 2^2 + \cdots + d_n \cdot n^2 = 0,
\vspace*{2mm}\\
\hspace*{27mm} \vdots \\
d_1 \cdot 1^n + d_2 \cdot 2^n + \cdots + d_n \cdot n^n = 0.
\end{array}
\right.
\end{align*}
Since the coefficient matrix of this system is a Vandermonde
matrix whose columns are multiplied by distinct nonzero
scalars, its determinant is nonzero.
Thus, the linear system admits only the trivial solution,
which implies $d_k = 0$ for all $k \in \{ 1, 2, \ldots, n \}$.

Consequently, equation (\ref{Enan:11}) simplifies directly to
the identity
\begin{align} \label{Enan:12}
\Gamma_{a,n} f(x)
= \sum_{j=1}^{a^n-n} c_{j,0} E_n f
  \big( \underbrace{x, \ldots, x}_{n}, (a^n-n-j+1) x \big).
\end{align}
By omitting the terms where $c_{j,0} = 0$ and rearranging the
index mapping, the right-hand side of (\ref{Enan:12}) can be
rewritten in the desired form
\begin{align*}
\Gamma_{a,n} f(x)
= \sum_{i=1}^m
  b_i E_n f(a_{i,1} x, a_{i,2} x, \ldots, a_{i,n+1} x),
\end{align*}
where $b_i$ and $a_{i,j}$ are nonzero integers for all
$i \in \{ 1, 2, \ldots, m \}$ and
$j \in \{ 1, 2, \ldots, n+1 \}$.
This completes the proof.
\hfill$\Box$


\begin{exam} \label{Enan:ex9}
The following identity illustrates the general relation
$(\ref{Enan:10})$ in Lemma \ref{Enan:lem8} for the specific
case where $a = 3$ and $n = 3$:
\begin{align*}
\Gamma_{3,3} f(x) = \sum_{k=1}^{24} c_k E_3 f(x,x,x,kx)
\end{align*}
for all $x \in V$, where the coefficients $c_k$ are given by
\begin{align*}
\begin{array}{llll}
c_{24} = 1,  & c_{18} = 28, & c_{12} = 91,  & c_{6} = 151,
\vspace*{2mm}\\
c_{23} = 3,  & c_{17} = 36, & c_{11} = 105, & c_{5} = 93,
\vspace*{2mm}\\
c_{22} = 6,  & c_{16} = 45, & c_{10} = 120, & c_{4} = -3,
\vspace*{2mm}\\
c_{21} = 10, & c_{15} = 55, & c_{9}  = 136, & c_{3} = -137,
\vspace*{2mm}\\
c_{20} = 15, & c_{14} = 66, & c_{8}  = 153, & c_{2} = -309,
\vspace*{2mm}\\
c_{19} = 21, & c_{13} = 78, & c_{7}  = 171, & c_{1} = -519.
\end{array}
\end{align*}
\end{exam}

Now, we present the main result of this paper, which
establishes the generalized Hyers-Ulam stability of the
functional equation $E_n f(x_1, x_2, \ldots, x_{n+1}) = 0$.


\begin{thm} \label{Enan:mthm}
Let $\ell \in \{ 0, 1, \ldots, n \}$.
If a function
$\phi : (V \setminus \{ 0 \})^{n+1} \to [0, \infty)$ satisfies
the convergence condition
\begin{align} \label{Enan:13}
\left\{
\begin{array}{ll}
{\displaystyle
\sum_{i=0}^\infty \frac{1}{a^i}
\phi \big( a^i x_1, a^i x_2, \ldots, a^i x_{n+1} \big)
< \infty}
& (\mbox{for}~ \ell = 0), \vspace*{2mm}\\
{\displaystyle
\sum_{i=0}^\infty \frac{1}{a^{(\ell+1) i}}
\phi \big( a^i x_1, a^i x_2, \ldots, a^i x_{n+1} \big)}
& (\mbox{for}~ 0 < \ell < n), \\
{\displaystyle \hspace*{5mm}
+\sum_{i=0}^\infty a^{\ell i}
\phi \bigg( \frac{1}{a^i} x_1, \frac{1}{a^i} x_2, \ldots,
            \frac{1}{a^i} x_{n+1}
     \bigg) < \infty} \vspace*{2mm}\\
{\displaystyle
\sum_{i=0}^\infty a^{ni}
\phi \bigg( \frac{1}{a^i} x_1, \frac{1}{a^i} x_2, \ldots,
            \frac{1}{a^i} x_{n+1}
     \bigg) < \infty}
& (\mbox{for}~ \ell = n)
\end{array}
\right.
\end{align}
for all $x_1, x_2, \ldots, x_{n+1} \in V \setminus \{ 0 \}$,
and a mapping $f : V \to Y$ satisfies the inequality
\begin{align} \label{Enan:14}
\| E_n f(x_1, x_2, \ldots, x_{n+1}) \|
\leq \phi(x_1, x_2, \ldots, x_{n+1})
\end{align}
for all $x_1, x_2, \ldots, x_{n+1} \in V \setminus \{ 0 \}$
with $f(0) = 0$, then there exists a unique mapping
$F : V \to Y$ satisfying $E_n F(x_1, x_2, \ldots, x_{n+1}) = 0$
for all $x_1, x_2, \ldots, x_{n+1} \in V$ and the inequality
$(\ref{Enan:9})$ for all $x \in V \setminus \{ 0 \}$, where 
the function $\varphi : V \setminus \{ 0 \} \to [0, \infty)$ 
is defined by
\begin{align*}
\varphi(x)
:= \sum_{i=1}^m
   | b_i | \phi(a_{i,1} x, a_{i,2} x, \ldots, a_{i,n+1} x)
\end{align*}
for all $x \in V \setminus \{ 0 \}$, and $m = a^n - n$, and
$b_i, a_{i,j}$ are the nonzero integers given in
$(\ref{Enan:10})$.
Furthermore, for each $k \in \{ 1, 2, \ldots, n \}$, the
component mapping $F_{k,a,n} : V \to Y$ of the mapping
$F(x) = \sum\limits_{k=1}^n F_{k,a,n}(x)$ is given by
$(\ref{Enan:8})$ for all $x \in V$.
\end{thm}

\noindent
\textbf{Proof.}
From (\ref{Enan:10}), (\ref{Enan:14}), and the definition of
$\varphi$, applying the triangle inequality yields
\begin{align} \label{Enan:15}
\begin{split}
\| \Gamma_{a,n} f(x) \|
& =    \Bigg\|
       \sum_{i=1}^m
       b_i E_n f(a_{i,1} x, a_{i,2} x, \ldots, a_{i,n+1} x)
       \Bigg\| \\
& \leq \sum_{i=1}^m
       | b_i | \phi(a_{i,1} x, a_{i,2} x, \ldots, a_{i,n+1} x)
       \\
& =    \varphi(x)
\end{split}
\end{align}
for all $x \in V \setminus \{ 0 \}$.

Since $f$ and $\varphi$ satisfy the required assumptions of
Theorem \ref{Enan:thm7} by virtue of condition (\ref{Enan:13})
and inequality (\ref{Enan:15}), there exists a unique mapping
$F : V \to Y$ satisfying $\Gamma_{a,n} F(x) = 0$ for all
$x \in V$, and inequality (\ref{Enan:9}) for all
$x \in V \setminus \{ 0 \}$.
In particular, for each $k \in \{ 1, 2, \ldots, n \}$, the
component mapping $F_{k,a,n} : V \to Y$ is uniquely determined
by (\ref{Enan:8}) for all $x \in V \setminus \{ 0 \}$, with
$F_{k,a,n}(0) = 0$.
Since $f(0) = 0$ implies $f_{k,a,n}(0) = 0$ for all
$k \in \{ 1, 2, \ldots, n \}$, the specification (\ref{Enan:8})
naturally extends to the entire space $V$, thereby ensuring
that $F(0) = 0$.

From (\ref{Enan:8}) and the definition of $f_{k,a,n}(x)$, we
observe that the following relations
\begin{align*}
& \| E_n F_{k,a,n}(x_1, x_2, \ldots, x_{n+1}) \| \\
& \hspace*{3mm}
  =    \lim_{i \to \infty} \frac{1}{a^{ki}}
       \big\|
       E_n f_{k,a,n} \big( a^i x_1, \ldots, a^i x_{n+1} \big)
       \big\| \\
& \hspace*{3mm}
  =    \lim_{i \to \infty} \frac{| \omega_{k,a,n} |}{a^{ki}}
       \Bigg\|
       E_n f
       \big( a^{i+n-1} x_1, \ldots, a^{i+n-1} x_{n+1} \big) \\
& \hspace*{8mm}
       +\sum_{m=1}^{n-1} (-1)^m
       \sum_{\substack{1 \leq i_1 < \cdots < i_m \leq n \\
                       i_1, \ldots, i_m \neq k}}
       a^{i_1 + \cdots + i_m}
       E_n f
       \big( a^{i+n-1-m} x_1, \ldots, a^{i+n-1-m} x_{n+1} \big)
       \Bigg\| \\
& \hspace*{3mm}
  \leq \lim_{i \to \infty} \frac{| \omega_{k,a,n} |}{a^{ki}}
       \Bigg(
       \phi \big(
            a^{i+n-1} x_1, \ldots, a^{i+n-1} x_{n+1}
            \big) \\
& \hspace*{8mm}
       +\sum_{m=1}^{n-1}
       \sum_{\substack{1 \leq i_1 < \cdots < i_m \leq n \\
                       i_1, \ldots, i_m \neq k}}
       a^{i_1 + \cdots + i_m}
       \phi
       \big( a^{i+n-1-m} x_1, \ldots, a^{i+n-1-m} x_{n+1} \big)
       \Bigg) \\
& \hspace*{3mm}
  =    0
\end{align*}
hold for all $x_1, \ldots, x_{n+1} \in V \setminus \{ 0 \}$
when $k > \ell$.
Similarly, for the case $k \leq \ell$, the relations
\begin{align*}
& \| E_n F_{k,a,n}(x_1, x_2, \ldots, x_{n+1}) \| \\
& \hspace*{3mm}
  =    \lim_{i \to \infty} a^{ki}
       \bigg\| E_n f_{k,a,n}
               \bigg(
               \frac{1}{a^i} x_1, \ldots, \frac{1}{a^i} x_{n+1}
               \bigg)
       \bigg\| \\
& \hspace*{3mm}
  =    \lim_{i \to \infty} | \omega_{k,a,n} | a^{ki}
       \Bigg\|
       E_n f \bigg(
             \frac{1}{a^i} x_1, \ldots, \frac{1}{a^i} x_{n+1}
             \bigg) \\
& \hspace*{8mm}
       +\sum_{m=1}^{n-1} (-1)^m
       \sum_{\substack{1 \leq i_1 < \cdots < i_m \leq n \\
                       i_1, \ldots, i_m \neq k}}
       a^{i_1 + \cdots + i_m}
       E_n f \bigg(
             \frac{a^{n-1-m}}{a^i} x_1, \ldots,
             \frac{a^{n-1-m}}{a^i} x_{n+1}
             \bigg)
       \Bigg\| \\
& \hspace*{3mm}
  \leq | \omega_{k,a,n} |
       \lim_{i \to \infty} a^{ki}
       \Bigg(
       \phi \bigg(
            \frac{1}{a^i} x_1, \ldots, \frac{1}{a^i} x_{n+1}
            \bigg) \\
& \hspace*{8mm}
       +\sum_{m=1}^{n-1}
       \sum_{\substack{1 \leq i_1 < \cdots < i_m \leq n \\
                       i_1, \ldots, i_m \neq k}}
       a^{i_1 + \cdots + i_m}
       \phi \bigg( \frac{a^{n-1-m}}{a^i} x_1, \ldots,
                   \frac{a^{n-1-m}}{a^i} x_{n+1}
            \bigg)
       \Bigg) \\
& \hspace*{3mm}
  =    0
\end{align*}
hold for all
$x_1, x_2, \ldots, x_{n+1} \in V \setminus \{ 0 \}$.

Since $F(x) = \sum\limits_{k=1}^n F_{k,a,n}(x)$ for all
$x \in V$, it follows by linearity that
\begin{align*}
& E_n F(x_1, x_2, \ldots, x_{n+1}) \\
& \hspace*{3mm}
  = \sum_{k=1}^\ell E_n F_{k,a,n}(x_1, x_2, \ldots, x_{n+1}) +
    \sum_{k=\ell+1}^n E_n F_{k,a,n}(x_1, x_2, \ldots, x_{n+1})
    \\
& \hspace*{3mm}
  = 0
\end{align*}
for all $x_1, x_2, \ldots, x_{n+1} \in V \setminus \{ 0 \}$.
Combined with the boundary conditions, including the relation
$E_n F(0, x_2, \ldots, x_{n+1}) = 0$ which directly follows
from $F(0) = 0$, we conclude that the identity
$E_n F(x_1, x_2, \ldots, x_{n+1}) = 0$ holds for all
$x_1, x_2, \ldots, x_{n+1} \in V$.
This completes the proof.
\hfill$\Box$
\vspace*{5mm}

The following theorem presents a specialized application of
Theorem \ref{Enan:mthm} for the concrete case where $a = 3$
and $n = 3$.


\begin{thm} \label{Enan:thm11}
Let $\ell$ be a fixed integer such that $0 \leq \ell \leq 3$.
If a function $\phi : (V \setminus \{ 0 \})^4 \to [0, \infty)$
satisfies the convergence condition
\begin{align*}
\left\{
\begin{array}{ll}
{\displaystyle
\sum_{i=0}^\infty \frac{1}{3^i}
\phi \big( 3^i x_1, 3^i x_2, 3^i x_3, 3^i x_4 \big) < \infty}
& (\mbox{for}~ \ell = 0), \vspace*{2mm}\\
{\displaystyle
\sum_{i=0}^\infty \frac{1}{3^{(\ell+1) i}}
\phi \big( 3^i x_1, 3^i x_2, 3^i x_3, 3^i x_4 \big)}
& (\mbox{for}~ \ell \in \{ 1, 2 \}), \\
{\displaystyle \hspace*{5mm}
+\sum_{i=0}^\infty 3^{\ell i}
\phi \bigg( \frac{1}{3^i} x_1, \frac{1}{3^i} x_2,
            \frac{1}{3^i} x_3, \frac{1}{3^i} x_4
     \bigg) < \infty} \vspace*{2mm}\\
{\displaystyle
\sum_{i=0}^\infty 3^{3i}
\phi \bigg( \frac{1}{3^i} x_1, \frac{1}{3^i} x_2,
            \frac{1}{3^i} x_3, \frac{1}{3^i} x_4
     \bigg) < \infty}
& (\mbox{for}~ \ell = 3)
\end{array}
\right.
\end{align*}
for all $x_1, x_2, x_3, x_4 \in V \setminus \{ 0 \}$, and a
mapping $f : V \to Y$ satisfies the inequality
\begin{align*}
\| E_3 f(x_1, x_2, x_3, x_4) \| \leq \phi(x_1, x_2, x_3, x_4)
\end{align*}
for all $x_1, x_2, x_3, x_4 \in V \setminus \{ 0 \}$ with
$f(0) = 0$, then there exists a unique mapping $F : V \to Y$
satisfying $E_3 F(x_1, x_2, x_3, x_4) = 0$ for all
$x_1, x_2, x_3, x_4 \in V$ and the inequality
\begin{align*}
\| f(x) - F(x) \| \leq \left\{
\begin{array}{ll}
{\displaystyle
\frac{1}{144} \sum_{i=0}^\infty \frac{\varphi(3^i x)}{3^{i+1}}}
& (\mbox{for}~ \ell = 0), \vspace*{2mm}\\
{\displaystyle
\frac{1}{144} \sum_{i=0}^\infty 3^i
\varphi \bigg( \frac{1}{3^{i+1}} x \bigg) +
\frac{1}{108} \sum_{i=0}^\infty
\frac{\varphi(3^i x)}{3^{2(i+1)}}}
& (\mbox{for}~ \ell = 1), \vspace*{2mm}\\
{\displaystyle
\frac{1}{108} \sum_{i=0}^\infty 3^{2i}
\varphi \bigg( \frac{1}{3^{i+1}} x \bigg) +
\frac{1}{432} \sum_{i=0}^\infty
\frac{\varphi(3^i x)}{3^{3(i+1)}}}
& (\mbox{for}~ \ell = 2), \vspace*{2mm}\\
{\displaystyle
\frac{1}{432} \sum_{i=2}^\infty 3^{3i}
\varphi \bigg( \frac{1}{3^{i+1}} x \bigg)}
& (\mbox{for}~ \ell = 3)
\end{array}
\right.
\end{align*}
for all $x \in V \setminus \{ 0 \}$, where the function
$\varphi : V \setminus \{ 0 \} \to [0,\infty)$ is defined by
\begin{align*}
\varphi(x) := \sum_{k=1}^{24} | c_k | \phi(x,x,x,kx)
\end{align*}
for all $x \in V \setminus \{ 0 \}$, and $c_k$ is the integer
specified in Example \ref{Enan:ex9}.
Furthermore, the unique mapping $F$ can be decomposed into
three component monomial mappings,
$F(x) = \sum\limits_{k=1}^{3} F_{k,3,3}(x)$ for all $x \in V$,
where each component $F_{k,3,3} : V \to Y$ is explicitly
defined by
\begin{align*}
F_{1,3,3}(x)
& = \left\{
    \begin{array}{ll}
    {\displaystyle
    \lim_{i \to \infty}
    \frac{1}{3^i} f_{1,3,3} \big( 3^i x \big)}
    & (\mbox{for}~ \ell = 0), \vspace*{2mm}\\
    {\displaystyle
    \lim_{i \to \infty}
    3^i f_{1,3,3} \bigg( \frac{1}{3^i} x \bigg)}
    & (\mbox{for}~ \ell \in \{ 1, 2, 3 \}),
    \end{array}
    \right. \vspace*{2mm}\\
F_{2,3,3}(x)
& = \left\{
    \begin{array}{ll}
    {\displaystyle
    \lim_{i \to \infty}
    \frac{1}{3^{2i}} f_{2,3,3} \big( 3^i x \big)}
    & (\mbox{for}~ \ell \in \{ 0, 1 \}), \vspace*{2mm}\\
    {\displaystyle
    \lim_{i \to \infty}
    3^{2i} f_{2,3,3} \bigg( \frac{1}{3^i} x \bigg)}
    & (\mbox{for}~ \ell \in \{ 2, 3 \}),
    \end{array}
    \right. \vspace*{2mm}\\
F_{3,3,3}(x)
& = \left\{
    \begin{array}{ll}
    {\displaystyle
    \lim_{i \to \infty}
    \frac{1}{3^{3i}} f_{3,3,3} \big( 3^i x \big)}
    & (\mbox{for}~ \ell \in \{ 0, 1, 2 \}), \vspace*{2mm}\\
    {\displaystyle
    \lim_{i \to \infty}
    3^{3i} f_{3,3,3} \bigg( \frac{1}{3^i} x \bigg)}
    & (\mbox{for}~ \ell = 3)
    \end{array}
    \right.
\end{align*}
with $F_{k,3,3}(0) = 0$.
\end{thm}

As a corollary to Theorem \ref{Enan:mthm}, one can easily
establish the Hyers-Ulam-Rassias stability of the functional
equation $E_n f(x_1, x_2, \ldots, x_{n+1}) = 0$ for all
$x_1, x_2, \ldots, x_{n+1} \in X \setminus \{ 0 \}$.


\begin{cor} \label{Enan:cor12}
Let $\theta$ be a positive real constant and let $p$ be a real
number such that $p \not\in \{ 1, 2, \ldots, n \}$.
If a mapping $f : X \to Y$ satisfies the inequality
\begin{align*}
\| E_n f(x_1, x_2, \ldots, x_{n+1}) \|
\leq \theta
     \big(
     \| x_1 \|^p + \| x_2 \|^p + \cdots + \| x_{n+1} \|^p
     \big)
\end{align*}
for all $x_1, x_2, \ldots, x_{n+1} \in X \setminus \{ 0 \}$
with $f(0) = 0$, then there exists a unique mapping
$F : X \to Y$ satisfying $E_n F(x_1, x_2, \ldots, x_{n+1}) = 0$
for all $x_1, x_2, \ldots, x_{n+1} \in X$ and
\begin{align*}
\| f(x) - F (x)\| \leq \left\{
\begin{array}{ll}
{\displaystyle
\frac{| \omega_{1,a,n} |}{a - a^p} M \theta \| x \|^p}
& (\mbox{for}~ p < 1), \vspace*{2mm}\\
{\displaystyle
\bigg( \frac{| \omega_{\ell,a,n} |}{a^p - a^\ell} +
       \frac{| \omega_{\ell+1,a,n} |}{a^{\ell+1} - a^p}
\bigg) M \theta \| x \|^p}
& (\mbox{for}~ \ell < p < \ell + 1), \vspace*{2mm}\\
{\displaystyle
\frac{a^{n(n-1)} | \omega_{n,a,n} |}{a^{(n-1) p} (a^p - a^n)}
M \theta \| x \|^p}
& (\mbox{for}~ p > n)
\end{array}
\right.
\end{align*}
for all $x \in X \setminus \{ 0 \}$ and some
$\ell \in \{ 1, 2, \ldots, n-1 \}$, where
\begin{align*}
M := \sum_{i=1}^m | b_i | \sum_{j=1}^{n+1} | a_{i,j} |^p,
\end{align*}
with $m = a^n - n$, and $b_i, a_{i,j}$ are the nonzero
integers given in $(\ref{Enan:10})$.
\end{cor}

\noindent
\textbf{Proof.}
By setting
$\phi(x_1, x_2, \ldots, x_{n+1})
 = \theta (\| x_1 \|^p + \| x_2 \|^p + \cdots + \| x_{n+1} \|^p)$
in Theorem \ref{Enan:mthm}, we directly obtain
$\varphi(x) = M \theta \| x \|^p$.
Furthermore, we observe that $\phi$ satisfies the convergence
condition (\ref{Enan:13}) for each corresponding value of
$\ell$ determined by the range of $p$: $\ell = 0$ for $p < 1$;
$\ell = \lfloor p \rfloor$ for $1 < p < n$ with
$p \not\in \{ 1, 2, \ldots, n \}$ (where $\lfloor p \rfloor$
denotes the greatest integer less than or equal to $p$); and
$\ell = n$ for $p > n$.
Consequently, applying inequality (\ref{Enan:9}) of Theorem
\ref{Enan:thm7} to each specified case yields the desired
error estimates.
\hfill$\Box$
\vspace{5mm}

The following example illustrates a concrete application of
Corollary \ref{Enan:cor12} for $a = 3$ and $n = 3$:


\begin{exam}
Let $\theta$ be a positive real constant and let $p$ be a real
number such that $p \not\in \{ 1, 2, 3 \}$.
If a mapping $f : X \to Y$ satisfies the inequality
\begin{align*}
\| E_3 f(x_1, x_2, x_3, x_4) \|
\leq \theta
     \big(
     \| x_1 \|^p + \| x_2 \|^p + \| x_3 \|^p + \| x_4 \|^p
     \big)
\end{align*}
for all $x_1, x_2, x_3, x_4 \in X \setminus \{ 0 \}$ with
$f(0) = 0$, then there exists a unique mapping $F : X \to Y$
satisfying $E_3 F(x_1, x_2, x_3, x_4) = 0$ for all
$x_1, x_2, x_3, x_4 \in X$ and the explicit approximation
inequality
\begin{align*}
\| f(x) - F(x) \| \leq \left\{
\begin{array}{ll}
{\displaystyle
\frac{1}{144 (3 - 3^p)} M_p \theta \| x \|^p}
& (\mbox{for}~ p < 1), \vspace*{2mm}\\
{\displaystyle
\bigg( \frac{1}{144 (3^p - 3)} + \frac{1}{108 (9 - 3^p)} \bigg)
M_p \theta \| x \|^p}
& (\mbox{for}~ 1 < p < 2), \vspace*{2mm}\\
{\displaystyle
\bigg( \frac{1}{108 (3^p - 9)} + \frac{1}{432 (27 - 3^p)} \bigg)
M_p \theta \| x \|^p}
& (\mbox{for}~ 2 < p < 3), \vspace*{2mm}\\
{\displaystyle
\frac{729}{432 \cdot 3^{2p} (3^p - 27)} M_p \theta \| x \|^p}
& (\mbox{for}~ p > 3)
\end{array}
\right.
\end{align*}
for all $x \in X \setminus \{ 0 \}$, where the constant
function $M_p$ as a function of the power $p$ is given by
\begin{align*}
M_p & = \sum_{k=1}^{24} | c_k | \big( 3 + k^p \big) \\
    & = 1 \cdot 24^p + 3 \cdot 23^p + 6 \cdot 22^p +
        10 \cdot 21^p + 15 \cdot 20^p + 21 \cdot 19^p \\
    & \hspace*{5mm}
        +28 \cdot 18^p + 36 \cdot 17^p + 45 \cdot 16^p +
        55 \cdot 15^p + 66 \cdot 14^p + 78 \cdot 13^p \\
    & \hspace*{5mm}
        +91 \cdot 12^p + 105 \cdot 11^p + 120 \cdot 10^p +
        136 \cdot 9^p + 153 \cdot 8^p + 171 \cdot 7^p \\
    & \hspace*{5mm}
        +151 \cdot 6^p + 93 \cdot 5^p + 3 \cdot 4^p +
        137 \cdot 3^p + 309 \cdot 2^p + 519 \cdot 1^p \\
    & \hspace*{5mm}
        +7056
\end{align*}
and $c_k$ is the integer specified in Example \ref{Enan:ex9}.
\end{exam}
\vspace*{5mm}

In particular, the case $p < 0$ in Corollary \ref{Enan:cor12}
yields the hyperstability of the functional equation
$E_n f(x_1, x_2, \ldots, x_{n+1}) = 0$.


\begin{thm} \label{Enan:thm14}
Let $\theta$ be a positive real constant and let $p$ be a
negative real number.
If a mapping $f : X \to Y$ satisfies the inequality
\begin{align*}
\| E_n f(x_1, x_2, \ldots, x_{n+1}) \|
\leq \theta
     \big(
     \| x_1 \|^p + \| x_2 \|^p + \cdots + \| x_{n+1} \|^p
     \big)
\end{align*}
for all $x_1, x_2, \ldots, x_{n+1} \in X \setminus \{ 0 \}$
with $f(0) = 0$, then $E_n f(x_1, x_2, \ldots, x_{n+1}) = 0$
for all $x_1, x_2, \ldots, x_{n+1} \in X$.
\end{thm}

\noindent
\textbf{Proof.}
In Corollary \ref{Enan:cor12}, the condition $p < 0$ implies
the case $p < 1$.
Thus, there exists a unique mapping $F : X \to Y$ satisfying
$F(0) = 0$, $E_n F(x_1, x_2, \ldots, x_{n+1}) = 0$, and
\begin{align*}
\| f(x) - F(x) \|
\leq \frac{| \omega_{1,a,n} |}{a - a^p} M \theta \| x \|^p
\end{align*}
for all $x \in X \setminus \{ 0 \}$.
Since $E_n F(x_1, x_2, \ldots, x_{n+1}) = 0$ for all
$x_1, x_2, \ldots, x_{n+1} \in X$, it follows from the
identity of $E_n f((1-m) x, m x, \ldots, m x)$ that
\begin{align*}
n (-1)^n (f(x) - F(x))
& = \sum_{\substack{i=0 \\ i \neq 1}}^n
    \binom{n}{i} (-1)^{i-n+1} (f - F)((1-m+i) x) \\
& \hspace*{5mm}
    -\sum_{i=1}^n \binom{n}{i} (-1)^{i-n+1} (f - F)(imx) \\
& \hspace*{5mm}
    -E_n f((1-m) x, m x, \ldots, m x)
\end{align*}
for all $x \in X \setminus \{ 0 \}$ and $m \in \mathbb{N}$
with $m \geq n + 2$.
By applying the triangle inequality, we obtain
\begin{align*}
n \| f(x) - F(x) \|
& \leq \sum_{\substack{i=0 \\ i \neq 1}}^n
       \binom{n}{i}  \| (f - F)((1-m+i) x) \| \\
& \hspace*{5mm}
       +\sum_{i=1}^n \binom{n}{i} \| (f - F)(imx) \| +
       \| E_n f((1-m) x, m x, \ldots, m x) \| \\
& \leq \frac{| \omega_{1,a,n} |}{a - a^p}
       \Bigg( \sum_{\substack{i=0 \\ i \neq 1}}^n
              \binom{n}{i} | 1-m+i |^p +
              \sum_{i=1}^n \binom{n}{i} | im |^p
       \Bigg)
       M \theta \| x \|^p \\
& \hspace*{5mm}
       +\theta \big( | 1-m |^p + n m^p \big) \| x \|^p
\end{align*}
for all $x \in X \setminus \{ 0 \}$.
Since $p < 0$, taking the limit as $m \to \infty$ on the
right-hand side yields
$\lim\limits_{m \to \infty} | 1-m+i |^p = 0$
($0 \leq i \leq n$, $i \neq 1$),
$\lim\limits_{m \to \infty} | im |^p = 0$, and
$\lim\limits_{m \to \infty} (| 1-m |^p + n m^p) = 0$.
Consequently, the right-hand side of the inequality approaches
$0$ as $m \to \infty$, which implies that $f(x) = F(x)$ for
all $x \in X \setminus \{ 0 \}$.

Combined with $f(0) = 0 = F(0)$, we obtain $f(x) = F(x)$ for
all $x \in X$.
Therefore, $f$ satisfies the functional equation
$E_n f(x_1, x_2, \ldots, x_{n+1}) = 0$ for all
$x_1, x_2, \ldots, x_{n+1} \in X$, which completes the proof.
\hfill $\Box$

\end{document}